\documentclass[11pt,preprint,times,3p,sort&compress]{elsarticle}

\usepackage[utf8]{inputenc} 

\usepackage{geometry} 
\usepackage{amsmath}
\usepackage{amsfonts}
\usepackage{amssymb}
\usepackage{amsthm}
\usepackage{enumerate}
\usepackage{hyperref}
\usepackage{tikz}
\usepackage{bbm}
\usepackage{xcolor}

\theoremstyle{plain}
\newtheorem{theorem}{Theorem}[section]
\newtheorem{lemma}[theorem]{Lemma}

\newtheorem{corollary}[theorem]{Corollary}

\theoremstyle{definition}
\newtheorem{definition}[theorem]{Definition}
\newtheorem{example}[theorem]{Example}

\theoremstyle{remark}
\newtheorem{remark}[theorem]{Remark}
\allowdisplaybreaks
\numberwithin{equation}{section}

\newcommand{\R}{\mathbb{R}}
\newcommand{\diag}{\mathrm{diag}}

\newcommand{\tr}{\mathrm{tr}}
\newcommand{\vol}{\mathrm{vol}}

\begin{document}

\begin{frontmatter}
\title{Volume Preservation by Runge-Kutta Methods}

\author[latrobe]{Philipp Bader} %
\ead{p.bader@latrobe.edu.au}
\author[latrobe]{David I. McLaren} %
\ead{d.mclaren@latrobe.edu.au}
\author[latrobe]{{G.R.W. Quispel}%
} \ead{r.quispel@latrobe.edu.au}
\author[damtp]{Marcus Webb\corref{cor1}} 
\ead{m.d.webb@maths.cam.ac.uk}
\cortext[cor1]{Corresponding author}

\address[latrobe]{
Department of Mathematics and Statistics, La Trobe University, 3086 Bundoora VIC, Australia}
\address[damtp]{DAMTP, University of Cambridge, Wilberforce Rd, Cambridge CB3 0WA, UK}


\begin{abstract}
It is a classical theorem of Liouville that Hamiltonian systems preserve volume in phase space. Any symplectic Runge-Kutta method will respect this property for such systems, but it has been shown that no B-Series method can be volume preserving for all volume preserving vector fields (BIT 47 (2007) 351–378 \& IMA J. Numer. Anal. 27 (2007) 381--405). In this paper we show that despite this result, symplectic Runge-Kutta methods can be volume preserving for a much larger class of vector fields than Hamiltonian systems, and discuss how some Runge-Kutta methods can preserve a modified measure exactly.
\end{abstract}
\begin{keyword}
volume preservation \sep Runge-Kutta method\sep measure preservation \sep Kahan's method
\end{keyword}
\end{frontmatter}




\section{Introduction}
The construction of numerical schemes for solving ordinary differential equations (ODEs) such that some qualitative geometrical property of the analytical solution is preserved exactly by the numerical solution is an area of great interest and active research today as part of the field of Geometric Integration. The most developed topic in this context is that of integrating Hamiltonian systems while preserving the symplecticity of the flow, and it was found that a class of Runge-Kutta (RK) methods, now called symplectic Runge-Kutta (SRK) methods, provides a convenient way to achieve this \cite[\S VI.4]{hairer2006geometric}. 

It is a classical theorem due to Liouville that Hamiltonian systems are also volume preserving: for all bounded open sets $\Omega$ of phase space, the flow map $\varphi_t$ satisfies $\vol(\varphi_t(\Omega))=\vol(\Omega)$ for all $t$. Equivalently, the Jacobian determinant, $\det(\varphi_t^\prime(x))$, is 1 for all $x$ and $t$ \cite[VI.9]{hairer2006geometric}. Any symplectic mapping of phase space has this property, and therefore SRK methods are volume preserving for Hamiltonian systems. Beyond Hamiltonian systems, an ODE $\dot{x} = f(x)$ is volume preserving if and only if $f$ is divergence free (sometimes called source free). General volume preservation like this can be found in applications involving incompressible fluid flows and vorticities, ergodic theory and statistical mechanics, and problems in electromagnetism \cite{iserles2007b,kang1995volume,quispel1995volume}.

One can ask if any SRK methods are volume preserving for all divergence free vector fields $f$, and it has been known for 20 years that the answer is no. Kang and Zai-Jiu showed that no RK method can be volume preserving even for the class of \emph{linear} divergence free vector fields \cite{kang1995volume}. It was later shown by Iserles, Quispel and Tse and independently by Chartier and Murua that no B-Series method can be volume preserving for all divergence free vector fields \cite{iserles2007b,chartier2005preserving}. However, Hairer, Lubich and Wanner have considered separable divergence free vector fields of the form
\begin{equation}\tag{HLW}
f(x,y) = (u(y),v(x))^\top,
\end{equation} 
for functions $u:\R^n \to \R^m$, $v:\R^m \to \R^n$ \cite[Thm.~9.4]{hairer2006geometric}. There the authors prove that any SRK method with at most two stages (and by the product rule of differentiation any composition of such methods) is volume preserving for these systems, giving a hint at the fact that SRK methods can be volume preserving for a much larger class of vector fields than just Hamiltonian systems.

As we will show in the introduction, vector fields $f$ in that class must satisfy the determinant condition
\begin{equation}\tag{det}\label{det}
\det\left(I+\frac{h}{2}f^\prime(x)\right) = \det\left(I-\frac{h}{2}f^\prime(x)\right) \text{ for all } h>0, \, x\in \R^n,
\end{equation}
where $I$ denotes the $n\times n$ identity matrix.
In order to substantiate this claim and in anticipation of some of the results to be discussed later, we consider the following three Runge-Kutta methods $x\mapsto\phi_h(x)$
which have been shown to preserve certain measures $\mu(x)dx$ for quadratic Hamiltonian vector fields  \cite{celledoni2013geometric}: 
\begin{enumerate}
\item The implicit midpoint rule
\[
\frac{\phi_h(x)-x}{h} = f\left(\frac{\phi_h(x)+x}{2}\right), \quad \text{with}\quad \mu(x)=1,
\]
\item the trapezoidal rule
\begin{equation}\label{trapezoidal}
\frac{\phi_h(x)-x}{h} = \tfrac12\Bigl(f\left(x\right)+f\left(\phi_h(x)\right)\Bigr), \quad \text{with}\quad \mu(x)=\det\left({1-\tfrac{h}{2}f'(x)}\right),
\end{equation}
\item and Kahan's method (restricted to quadratic vector fields)
\begin{equation}\label{kahan}
\frac{\phi_h(x)-x}{h} = 2f\left(\frac{x+\phi_h(x)}2\right)-\tfrac12f(x)-\tfrac12f(\phi_h(x)), \quad \text{with}\quad \mu(x)=\det\Bigl({1-\tfrac{h}{2}f'(x)}\Bigr)^{-1}.
\end{equation}
\end{enumerate}
These quadratic Hamiltonian vector fields satisfy the determinant condition \eqref{det} and we will establish in section~\ref{sec.measure} that this condition is essential for these measure preservation properties.
Indeed, using the chain rule, we compute the Jacobian matrix of the midpoint rule to 
\[
\phi_h^\prime(x) = I + \frac{h}{2}f^\prime\left(\frac{x+\phi_h(x)}{2}\right)\Bigl(I + \phi_h^\prime(x)\Bigr),
\]
which in turn gives the condition for volume preservation
\[
\det(\phi_h^\prime(x)) = \frac{\det(I+\frac{h}{2}f^\prime\bigl(({x+\phi_h(x)})/{2})\bigr)}{\det(I-\frac{h}{2}f^\prime\bigl((x+\phi_h(x))/2)\bigr)} = 1.
\]
Note that in agreement with \cite{kang1995volume}, it is clear that for the implicit midpoint rule we cannot consider a class of vector fields any larger than this and realistically expect volume preservation. Hence we restrict our discussion to vector fields satisfying this determinant condition \eqref{det}.
These functions, as we show later, are divergence free and include Hamiltonian systems and HLW separable systems described above. 

The contributions of this paper are to highlight the relevance of the determinant condition \eqref{det} for volume preservation by Runge-Kutta methods, and to introduce and prove results regarding volume preservation for some classes of vector fields lying \emph{between} Hamiltonian vector fields and those satisfying the determinant condition \eqref{det}. Not only does this further the understanding of Runge-Kutta methods and volume preservation of numerical methods in general, but it gives examples of where in applications one could in principle use Runge-Kutta methods and preserve volume for a non-Hamiltonian system. Furthermore, we discuss how Runge-Kutta methods can also preserve a modified measure exactly. 
The importance of such methods is that the dynamics of the numerical solution lie in the class of measure preserving systems, giving a qualitative advantage over methods lacking this property \cite{mclachlan2001kinds}.
It should be noted that there are general approaches to constructing volume preserving splitting methods for a general divergence free vector field \cite{hairer2006geometric, quispel1995volume,kang1995volume}, but Runge-Kutta methods offer practical and theoretical simplicity and familiarity.

\section{Properties of Runge-Kutta methods}

This section is fairly technical, but it provides us with the necessary tools for the discussion in sections 3 and 4. We use the following notation to describe a Runge-Kutta method for the autonomous system $\dot{x}=f(x)$. We assume $f$ is continuously differentiable throughout the paper. For each step-size $h$, a $s$-stage Runge-Kutta method provides a map $\phi_h : \R^n \to \R^n$, defined by
\begin{equation*}
\phi_h(x) = x + h\sum_{i = 1}^s b_i f(k_i),
\end{equation*}
where the stages $k_i$ satisfy
\begin{equation*}
k_i = x + h\sum_{j=1}^s a_{ij}f(k_j), \text{ for } i= 1, \ldots, s.
\end{equation*}
As usual, we consolidate the $b_i$\rq{}s and $a_{ij}$\rq{}s into the Butcher tableau consisting of the vector $b$ and the matrix $A$. We make use of the Kronecker product throughout, which for $A \in \R^{n\times n}$ and $B \in \R^{m\times m}$ is defined to be
\begin{equation*}
A\otimes B = \begin{pmatrix} a_{11} B & \cdots & a_{1n} B \\ \vdots & \ddots & \vdots \\ a_{n1} B & \cdots & a_{nn} B \end{pmatrix}\in\R^{nm\times nm}.
\end{equation*}
\begin{lemma}\label{basicdeterminantlemma}
The Jacobian matrix of a RK method can be written as
\begin{equation}\label{rungeJacobianeqn}
\phi_h^\prime(x) = I + h(b^\top \otimes I) F (I_s\otimes I - h(A\otimes I)F)^{-1}(\mathbbm{1}\otimes I),
\end{equation}
with determinant
\begin{equation}\label{rungeJacobiandet}
\det(\phi_h'(x)) = \frac{\det(I_s\otimes I - h((A-\mathbbm{1}b^\top)\otimes I)F)}{\det(I_s \otimes I - h(A\otimes I)F)},
\end{equation}
where $F = \diag(f'(k_1),\ldots,f'(k_s))$, $\mathbbm{1}$ is an $s\times 1$ vector of 1\rq{}s and $I_s$ is the $s\times s$ identity matrix.
\begin{proof}
Computing directly, we find
\begin{equation}\label{rungeJacobian}
\phi_h'(x) = I + h \sum_{i=1}^s b_i f'(k_i) k_i'(x) = I + h(b^\top \otimes I) F (k_1'(x),\ldots,k_s'(x))^\top.
\end{equation}
By definition of the stages $k_i$, the derivatives $k_i'(x)$ satisfy
\begin{equation*}
\begin{pmatrix}
          I - ha_{11}f'(k_1) & -ha_{12}f'(k_2) & \cdots & -ha_{1s}f'(k_s) \\
          -ha_{21}f'(k_1) & I - ha_{22}f'(k_2) & \cdots & -ha_{2s}f'(k_s) \\
           \vdots & \vdots & \ddots & \vdots \\
           -ha_{s1}f'(k_1) & -ha_{s2}f'(k_2) & \cdots  & I-ha_{ss}f'(k_s) \end{pmatrix}\begin{pmatrix} k_1'(x) \\ k_2'(x) \\ \vdots \\ k_s'(x) \end{pmatrix} = \begin{pmatrix} I \\ I \\ \vdots \\ I \end{pmatrix}.
\end{equation*}
Written more compactly using  Kronecker products, this is
\begin{equation}\label{krelation}
(I_s\otimes I - h(A\otimes I)F) (k_1'(x),\ldots,k_s'(x))^\top = \mathbbm{1}\otimes I.
\end{equation}
The form of the Jacobian matrix can now be found by substituting \eqref{krelation} into \eqref{rungeJacobian}.

For the determinant, use the block determinant identity
\begin{equation}\label{blockdetidentity}
\det(U)\det(X-WU^{-1}V)=\det\begin{pmatrix} U & V \\ W & X \end{pmatrix} = \det(X)\det(U-VX^{-1}W)
\end{equation}
on the expression \eqref{rungeJacobianeqn} with $U = I_s\otimes I - h(A\otimes I)F$, $V = (\mathbbm{1}\otimes I)$, $W = -h(b^\top \otimes I) F$ and $X = I$.
\end{proof}
\end{lemma}

We wish to understand for which vector fields $f$ and which Runge-Kutta methods defined by $A$ and $b$, the determinant \eqref{rungeJacobiandet} is unity. As one might expect, this turns out to be simpler for symplectic Runge-Kutta methods. Now, for the purpose of exposition, we restrict to methods described in the following definition and instruct the reader in how certain results can be proven for general SRK methods at the end of the section.

\begin{definition}
A SRK method is said to be a \emph{special symplectic Runge-Kutta method (SSRK)} if $b_j \neq 0$ for all $j$, so that the Butcher tableau may be written $A = \frac{1}{2}(\Omega+\mathbbm{1}\mathbbm{1}^\top)B$, where $B = \diag(b)$ and $\Omega$ is a skew-symmetric matrix.
\end{definition}
This definition is reasonable because if $b_j \neq 0$ for all $j$, then the matrix $M = BA + A^\top B - bb^\top$ is zero (which implies the method is symplectic) if and only if $\Omega$ is skew-symmetric. The expression $\frac{1}{2}(\Omega+\mathbbm{1}\mathbbm{1}^\top)B$ therefore constitutes a normal form for most SRK methods of interest \cite{hairer2006geometric}.
\begin{lemma}\label{SSRK}
An s-stage SSRK method is volume preserving for $\dot{x}=f(x)$ if and only if
\begin{equation}\label{sstagecondition}
\det(I_s\otimes I - h(A\otimes I)F) = \det(I_s\otimes I + h(A^\top \otimes I)F),
\end{equation}
where $F = \diag(f'(k_1),\ldots,f'(k_s))$.
\begin{proof}
The equation $M=0$ can be written $-A^\top = B(A-\mathbbm{1}b^\top)B^{-1}$. Hence
\begin{align*}
\det(I_s\otimes I + h(A^\top \otimes I)F) &= \det(I_s\otimes I - h(B(A-\mathbbm{1}b^\top)B^{-1}\otimes I)F) \\
                                                            &= \det(I_s \otimes I - h (B\otimes I)(A-\mathbbm{1}b^\top)\otimes I)F(B\otimes I)^{-1}) \\
                                                            &=\det(I_s\otimes I - h((A-\mathbbm{1}b^\top)\otimes I)F)
\end{align*}
The result now follows from Lemma \ref{basicdeterminantlemma}.
\end{proof}
\end{lemma}

When $s=1$, the only SSRK method is the implicit midpoint rule. In this case, Lemma \ref{SSRK} gives the determinant condition \eqref{det} from the introduction.

When $s=2$, we have a three-parameter family of SSRK methods, which reduces to two-parameter if we impose the consistency condition $b_1+b_2 = 1$. Now Lemma \ref{SSRK} gives the condition
\begin{equation}
\det\begin{pmatrix} I - ha_{11}f^\prime(k_1) & -ha_{12}f^\prime(k_2) \\ -ha_{21}f^\prime(k_1) & I-ha_{22}f^\prime(k_2) \end{pmatrix} = \det\begin{pmatrix} I + ha_{11}f^\prime(k_1) & ha_{21}f^\prime(k_2) \\ ha_{12}f^\prime(k_1) & I+ha_{22}f^\prime(k_2) \end{pmatrix}.
\end{equation}
Applying the block determinant identity \eqref{blockdetidentity}, this boils down to
\begin{multline}\label{2stagecond}
\det(I - ha_{11}f^\prime(k_1) - ha_{22}f^\prime(k_2) + h^2\det(A)f^\prime(k_1)f^\prime(k_2)) \\
= \det(I + ha_{11}f^\prime(k_1) + ha_{22}f^\prime(k_2) + h^2 \det(A) f^\prime(k_1) f^\prime(k_2)).
\end{multline}
We were able here to simplify the identity \eqref{blockdetidentity} because the top-left block ($I-ha_{11}f^\prime(k_1)$) and the bottom-left block ($-ha_{21}f^\prime(k_1)$) commute. This cannot be done for $s \geq 3$.

These next three lemmata give some basic operations that can be performed on the vector field which send volume preserving ODEs to volume preserving ODEs, and effect a simple change in the Jacobian determinant of some RK methods for general vector fields.

\begin{lemma}\label{linearchangeofvariables}
Let $f: \R^n \to \R^n$ and define a linear change of variables $\tilde{f}(x) = Pf(P^{-1}x)$ for some  invertible matrix $P$. Then the RK map $\tilde\phi_h$ for solving $\dot{x}=\tilde{f}(x)$ satisfies
\begin{equation}
\tilde\phi_h(x) = P\phi_h(P^{-1}x), \quad \tilde\phi_h^\prime(x) = P\phi_h^\prime(P^{-1}x)P^{-1}.
\end{equation}
\end{lemma}
\begin{lemma}\label{foliationlemma}
Let $u : \R^m \to \R^m$, $v : \R^{n+m} \to \R^n$ and define $f:\R^{m+n} \to \R^{m+n}$ by
\begin{equation}\label{prototypefoliation}
f\left(x, y \right) = \begin{pmatrix} u(x) \\ v(x,y)\end{pmatrix}.
\end{equation}
Now let $\phi_h:\R^{n+m}\to\R^{n+m}$ be a \emph{one-stage} RK map for solving $(\dot{x},\dot{y})^\top=f(x,y)$, $\psi_h : \R^m\to\R^m$ that for solving $\dot{x} = u(x)$, and $\chi_h : \R^{n+m} \to \R^n$ that for solving $\dot{y} = v(x,y)$ where $x$ is treated as a parameter. Then
\begin{equation}\label{1stagerk}
\phi_h(x,y) = \begin{pmatrix} \psi_h(x)\\ \chi_h(k_1(x),y) \end{pmatrix},
\end{equation}
and consequently
\begin{equation*}
\det(\phi_h'(x,y)) = \det(\psi_h'(x))\det(\partial_y\chi_h(k_1(x),y)),
\end{equation*}
where $k_1(x)=x+ha_{11} u(k_1(x))$ is the internal stage of the RK method $\psi_h(x)$ and $\partial_y$ denotes the derivative with respect to the $y$ coordinate.
\begin{proof}
The full method is
\begin{align}\label{eq.1stagefull}
\phi_h(x,y) &= \begin{pmatrix} x \\ y \end{pmatrix} + hb_1\begin{pmatrix}
u(k_1(x)) \\ v(k_1(x),l_1(k_1(x),y)) \end{pmatrix},\\
\intertext{with the internal stages}
\nonumber
\begin{pmatrix} k_1(x) \\ l_1(k_1(x),y) \end{pmatrix} &= \begin{pmatrix}x \\ y \end{pmatrix}+ ha_{11}\begin{pmatrix} u(k_1(x)) \\ v(k_1(x),l_1(k_1(x),y)) \end{pmatrix}.
\end{align}
The methods applied to each component of the $\R^{n+m}$ dimensional system are given by
\begin{align}\label{eq.1stagesep}
\begin{pmatrix} \psi_h(x) \\ \chi_h(x,y) \end{pmatrix} &= \begin{pmatrix} x \\ y \end{pmatrix} + hb_1\begin{pmatrix} u(k_1(x)) \\ v(x,l_1(x,y)) \end{pmatrix}.
\end{align}
Comparing \eqref{eq.1stagefull} with \eqref{eq.1stagesep} yields the result
\eqref{1stagerk}.
To prove the last part, note that the Jacobian matrix has block structure
\begin{equation*}
\phi_h^\prime(x,y) = \begin{pmatrix} \psi_h^\prime(x) & 0 \\ \partial_x(\chi_h(k_1(x),y)) & \partial_y(\chi_h(k_1(x),y)) \end{pmatrix}
\end{equation*}
and so the determinant $\det(\phi_h'(x,y))$ is the product of the determinants of the diagonal blocks.
\end{proof}
\end{lemma}

For some simple vector fields, this can be generalized to certain $s$-stage methods. Note that the notation for $\chi_h$ is different to that for Lemma \ref{foliationlemma}.
\begin{lemma}\label{foliationsum}
Let $u : \R^m \to \R^m$, $v : \R^{n} \to \R^n$, $w:\R^m\to\R^n$ and define $f:\R^{m+n} \to \R^{m+n}$ by
\begin{equation}
f\left(x, y \right) = \left(\begin{array}{l} u(x) \\ w(x) + v(y)\end{array} \right).
\end{equation}
Now let $\phi_h(x,y)$ be the RK map for solving $(\dot{x},\dot{y})^\top=f(x,y)$, $\psi_h(x)$ that for solving $\dot{x} = u(x)$, and $\chi_h(c,y)$ that for solving $\dot{y} = c + v(y)$. Define $c_i = \sum_j a_{ij}$. If the Butcher tableau is such that
\begin{equation}
\label{2stagecondition}
\delta_j(i,k) = \frac{a_{ij}-a_{kj}}{c_i-c_k} \quad 1\leq i,j,k\leq s,
\end{equation}
is finite and independent of distinct $i$ and $k$ for each $j$ then there exist functions $d_h,e_h,c_h:\R^m \to \R^n$ such that
\begin{equation}
\phi_h(x,y) = \begin{pmatrix} \psi_h(x) \\ \chi_h(d_h(x),y+he_h(x)) + hc_h(x) \end{pmatrix}  \text{ for all } y.
\end{equation}
Consequently,
\begin{equation}
\det(\phi_h'(x,y)) = \det(\psi_h'(x))\det(\partial_y\chi_h(d_h(x),y+he_h(x))).
\end{equation}
\begin{proof}
Write $\phi_h(x,y) = (\psi_h(x),\sigma_h(x,y))$. Note that $\sigma_h(x,y) \neq \chi_h(w(x),y)$, but they are related as follows.
\begin{align} \label{eq.sigmasum}
\sigma_h(x,y) &= y  + h\sum_{i=1}^s b_i w(k_i) +h\sum_{i=1}^s b_iv(l_i(w(k_1),\ldots,w(k_s),y)), \\ \nonumber
\chi_h(c,y) &= y + h\sum_{i=1}^sb_ic + h\sum_{i=1}^s b_iv(l_i(c,\ldots,c,y)), 
\end{align}
with stage values
\begin{align*}
 l_i(\zeta_1,\ldots,\zeta_s,y) &=  y  + h\sum_{j=1}^s a_{ij} \zeta_j + h\sum_{j=1}^s a_{ij}v(l_j(\zeta_1,\ldots,\zeta_s,y)).
\intertext{%
Now let $d$ be an arbitrary number. Then we have for each $i$,}
l_i(w(k_1),\ldots,w(k_s),y) &= y + h e_i + h\sum_{j=1}^s a_{ij} d + h\sum_{j=1}^s a_{ij}v(l_j(w(k_1),\ldots,w(k_s),y)),
\end{align*}
where $e_i = \sum_{j=1}^s a_{ij}(w(k_j)-d)$. Hence
\begin{equation}\label{stageidentity}
l_i(w(k_1),\ldots,w(k_s),y) = l_i(d,\ldots,d,y+he_i).
\end{equation}
We want to choose $d$ such that $e_i=e_k \forall i,k$. Equivalently,
\begin{equation*}
\sum_{j=1}^s a_{ij}(w(k_j)-d) = \sum_{j=1}^s a_{kj}(w(k_j)-d) \text{ for all } i \neq k.
\end{equation*}
Solving for $d$ we find
\begin{equation*}
d = \sum_{j=1}^s w(k_j)\left(\frac{a_{ij}-a_{kj}}{c_i-c_k}\right) \text{ for all } i \neq k.
\end{equation*}
This will only give us a unique finite value of $d$ no matter what values $w(k_i)$ take if the value of $\delta_j(i,k)$ is finite and independent of distinct $i$ and $k$ for every $j$, which is given by assumption. Hence we can set $d_h(x) = d$, $e_h(x) = e_1$ and by \eqref{stageidentity}, we write \eqref{eq.sigmasum} as
\begin{align*}
\sigma_h(x,y) &= y + h\sum_{i=1}^sb_iw(k_i) + h\sum_{i=1}^sb_iv(l_i(d_h,\ldots,d_h,y+he_h)) \\
                      &= (y + he_h) +h\sum_{i=1}^s b_i d_h  + h\sum_{i=1}^sb_iv(l_i(d_h,\ldots,d_h,y+he_h)) + h\left(\sum_{i=1}^s b_i(w(k_i)-d_h)-e_h\right) \\
                      &= \chi_h(d_h(x),y+he_h(x)) + hc_h(x),
\end{align*}
where $c_h(x) = \left(\sum_{i=1}^s b_i(w(k_i)-d_h)-e_h\right)$. The factorisation of the determinant is evident from the block structure of the Jacobian matrix  
\begin{equation*}
\phi_h^\prime(x,y) = \begin{pmatrix} \psi_h^\prime(x) & 0 \\ \star & \partial_y( \chi_h(d_h(x),y+he_h(x)) + hc_h(x) ) \end{pmatrix}.
\end{equation*}
\end{proof}
\end{lemma}

\begin{remark}
 Let us shed some light on the meaning of \eqref{2stagecondition} being finite and independent of distinct $i$ and $k$ for each $j$. The finiteness implies that the method has $c_i \neq c_k$ for all $i \neq k$, which is known as \emph{nonconfluency} \cite{hairer2006geometric}. 
For two-stage SSRK methods, condition \eqref{2stagecondition} is satisfied if the method is consistent and $\Omega\neq0$. 
There is a one-parameter family of self-adjoint three-stage SSRK methods of order four that satisfy the condition. The three-stage Gauss-Legendre method, however, does not belong to this class.
\end{remark}

%
\begin{definition}\label{linearfoliation}
A vector field $f:\R^{n+m} \to \R^{n+m}$ possesses a \emph{linear foliation} if there exists a linear change of variables as in Lemma \ref{linearchangeofvariables} such that $f$ is as in \eqref{prototypefoliation} from Lemma \ref{foliationlemma} for some functions $u$ and $v$. 
{Such vector fields are called \emph{linearly foliate}.}
See \cite{mclachlan2003lie} for general Lie group foliations in the context of Geometric Integration.
\end{definition}
\begin{remark}\label{SRK}
For general SRK methods, the condition in Lemma \ref{SSRK} along with the condition with $A$ replaced by $A-\mathbbm{1}b^\top$ is sufficient for volume preservation. This result can be obtained along the lines of \cite[Thm.~9.4]{hairer2006geometric} regarding separable systems (HLW), as follows. Consider the foliation $\dot{x} = f(x)$, $\dot{y} = -f^\prime(x)^\top y$, which is Hamiltonian with respect to $H(x,y) = y^\top f(x)$. Then, using the notation of Lemma \ref{foliationsum}, the Jacobian matrix of the Runge-Kutta map has block structure $\left(\begin{array}{cc} \phi_h^\prime(x) & 0 \\ \star & \partial_y \sigma_h(x,y) \end{array} \right)$. As in \cite[Thm.~9.4]{hairer2006geometric}, since the vector field is Hamiltonian, a SRK method will produce a symplectic map, which implies $\det(\phi_h^\prime(x))\det(\partial_y \sigma_h(x,y)) = 1$. Hence to show that $\det(\phi_h^\prime(x))=1$ it suffices to show that $\det(\phi_h^\prime(x)) = \det(\partial_y\sigma_h(x,y))$. Computing these two sides as in Lemma \ref{basicdeterminantlemma}, using the block determinant relation and equating numerators and denominators, gives the 2 conditions mentioned above.
\end{remark}

\section{Classification of volume preserving vector fields}

\begin{definition}
Define the following classes of vector fields on Euclidean space recursively using vector fields $f(x,y)=(u(x),v(x,y))^\top$ possessing linear foliations as in Definition~\ref{linearfoliation}.
\begin{align*}
\mathcal{H} &= \left\{ \text{$f$ such that there exists } P \text{ such that for all } x, Pf^\prime(x)P^{-1}=-f^\prime(x)^\top \right\}, \\
\mathcal{S} &= \left \{ \text{$f$ such that there exists } P \text{ such that for all  } x, Pf^\prime(x)P^{-1}=-f^\prime(x) \right\}, \\
 \mathcal{F}^{(\infty)} &= \left\{ f(x,y)=(u(x),v(x,y))^\top \text{ where } u\in\mathcal{H}\cup\mathcal{F}^{(\infty)} \text{ and there exists } P \text{ such that for all } x, y\right.
\\ & \left.
\qquad P\partial_y v(x,y)P^{-1}=-\partial_y v(x,y)^\top 
\right\}, \\
\mathcal{F}^{(2)} &=\left\{ 
f(x,y)=(u(x),v(x,y))^\top \text{ where } 
u\in\mathcal{S}\cup\mathcal{H}\cup\mathcal{F}^{(2)} 
\text{ and there exists } P \text{ such that for all } x, y\right.
\\ & \qquad \left. \text{ either } 
P\partial_y v(x,y)P^{-1}=-\partial_y v(x,y)^\top 
\text{ or }
P\partial_y v(x,y)P^{-1}=-\partial_y v(x,y)
\right\}, \\
 \mathcal{D} &= \left\{ \text{vector fields satisfying } \det(I+\frac{h}{2}f^\prime(x))=\det(I-\frac{h}{2}f^\prime(x)) \text{ for all } h>0 \text{ and all }x \right\}.
\end{align*}
\end{definition}

\begin{lemma}\label{Hlemma}
 The set $\mathcal{H}$ contains all vector fields of the form $f(x) = J^{-1}\nabla{H}(x)$ where $J$ is constant and skew-symmetric. All SRK methods are volume preserving for vector fields in $\mathcal{H}$.
 \begin{proof}
  For the first part, note that if $f(x) = J^{-1}\nabla{H}(x)$, then $Jf^\prime(x)J^{-1} = \nabla^2H(x)J^{-1} = - f^\prime(x)^\top$. For the second part, let $A \in \R^{s\times s}$ and $P$ be such that for all $x$, $Pf'(x)P^{-1} = -f'(x)^\top$. Then using the notation of Lemma \ref{SSRK},
  \begin{eqnarray}
   \det(I_s\otimes I - h(A\otimes I)F) &=& \det(I_s\otimes I - h(I_s\otimes P)(A \otimes I)(I_s \otimes P^{-1})(I_s\otimes P)F(I_s\otimes P^{-1})) \\
                                       &=& \det(I_s\otimes I + h(A\otimes I)F^\top) \\
                                       &=& \det(I_s\otimes I + hF(A^\top\otimes I) \text{ (transpose)} \\
                                       &=& \det(I_s \otimes I + h(A^\top\otimes I) F) \text{ (Sylvester's law)}.
  \end{eqnarray}
By Lemma \ref{SSRK} and Remark \ref{SRK}, all SRK methods are volume preserving.
 \end{proof}
\end{lemma}

\begin{lemma}\label{Slemma}
The set $\mathcal{S}$ contains all separable HLW systems i.e.~$f(x,y) = (u(y),v(x))^\top$. All SRK methods with at most 2 stages, and compositions thereof, are volume preserving for vector fields in $\mathcal{S}$. 
\begin{proof}
For the first part, note that if $f(x,y) = (u(y),v(x))^\top$, then $Df^\prime(x,y)D^{-1} = -f^\prime(x,y)$ where $D = \diag(I_m,-I_n)$. For the second part, let $A \in \R^{2\times2}$ and $P$ be such that for all $x$ $Pf'(x)P^{-1} = -f'(x)$. Then for the two stages $k_1$, $k_2$ of the SRK method,
\begin{multline*}
\det(I - ha_{11}f^\prime(k_1) - ha_{22}f^\prime(k_2) + h^2\det(A)f^\prime(k_1)f^\prime(k_2)) \\
= \det(I - ha_{11}Pf^\prime(k_1)P^{-1} - ha_{22}Pf^\prime(k_2)P^{-1} + h^2 \det(A) Pf^\prime(k_1)P^{-1} Pf^\prime(k_2)P^{-1}) \\
= \det(I + ha_{11}f^\prime(k_1) + ha_{22}f^\prime(k_2) + h^2\det(A)f^\prime(k_1)f^\prime(k_2)).
\end{multline*}
By \eqref{2stagecond} and Remark \ref{SRK}, all 2-stage SRK methods are volume preserving. To complete the proof, note that a 1-stage SRK method is equivalent to a 2-stage SRK method with two equal stages, and compositions of volume preserving maps are also volume preserving.
\end{proof}
\end{lemma}

\begin{lemma}\label{lemmaSets}
The inclusions $\mathcal{H} \subset \mathcal{F}^{(\infty)} \subset \mathcal{F}^{(2)} \subset \mathcal{D}$ and $\mathcal{S} \subset \mathcal{F}^{(2)} \subset \mathcal{D}$ hold.
\begin{proof}
$\mathcal{H} \subset \mathcal{F}^{(\infty)} \subset \mathcal{F}^{(2)}$ and $\mathcal{S} \subset \mathcal{F}^{(2)}$ are clear by considering trivial foliations in which $n+m=m$.  We will show that $\mathcal{S}\subset\mathcal{D}$, $\mathcal{H}\subset\mathcal{D}$ and that $\mathcal{D}$ is closed under the employed recursive process leading to linearly foliate systems.

For $f \in \mathcal{S}$, $\det(I+\frac{h}{2}f^\prime(x))=\det(I+\frac{h}{2}Pf^\prime(x)P^{-1}) = \det(I-\frac{h}{2}f^\prime(x))$. 

For $f \in \mathcal{H}$, $\det(I+\frac{h}{2}f^\prime(x))=\det(I+\frac{h}{2}Pf^\prime(x)P^{-1}) = \det(I-\frac{h}{2}f^\prime(x)^\top)=\det(I-\frac{h}{2}f^\prime(x))$. 

Let $f \in \mathcal{D}$ and define $\tilde{f}(x)=Pf(P^{-1}x)$ for an invertible matrix $P$. Then $\det(I+\frac{h}{2}\tilde{f}^\prime(x)) = \det(I+\frac{h}{2}Pf^\prime(P^{-1}x)P^{-1}) = \det(I+\frac{h}{2}f^\prime(P^{-1}x)$. Doing the same with a $-$ instead of a $+$ shows that $\tilde{f} \in\mathcal{D}$.

Let $f(x,y) = (u(x),v(x,y))^\top$ where $u \in\mathcal{D}$ and $y\mapsto v(x,y) \in \mathcal{D}$ for all $x$. Then
\begin{eqnarray}
\det(I+\frac{h}{2}f^\prime(x,y)) &=& \det\begin{pmatrix} I +\frac{h}{2} u^\prime(x) & 0 \\ \frac{h}{2}\partial_xv(x,y) & I + \frac{h}{2}\partial_yv(x,y) \end{pmatrix} \\
&=& \det(I+\frac{h}{2}u^\prime(x))\det(I+\frac{h}{2}\partial_yv(x,y)).
\end{eqnarray}
Doing the same with a $-$ instead of a $+$ shows that $f \in \mathcal{D}$.
\end{proof}
\end{lemma}\begin{figure}\centering
\begin{tikzpicture}[yscale=1.3]
\draw[rounded corners=5mm] (0,0) rectangle (15,2.5);
\draw[rounded corners=5mm] (.3,.2) rectangle (12,2.3);
\draw[rounded corners=3mm] (.6,.4) rectangle (9,2.1);
\draw[rounded corners=3mm] (.9,.6) rectangle (6,1.9);
\draw[rounded corners=3mm,fill=gray!20] (3,.8) rectangle (11,1.3);
\node at (3.5,1.6) {$\mathcal{H}$};
\node at (7.5,1.6) {$\mathcal{F}^{(\infty)}$};
\node at (10.5,1.6) {$\mathcal{F}^{(2)}$};
\node at (13.5,1.6) {$\mathcal{D}$};
\node at (7,1.05) {$\mathcal{S}$};
\end{tikzpicture}
\caption{Venn diagramm illustrating the relationships established by Lemma~\ref{lemmaSets}.}
\end{figure}
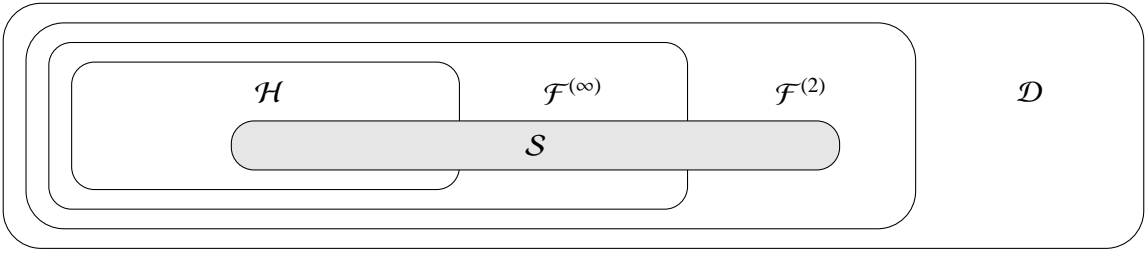
\begin{theorem}\label{dettheorem}
The following are equivalent.
\begin{enumerate}[(i)]
\item $f \in \mathcal{D}$
\item $\det(I+zf^\prime(x)) = \det(I-zf^\prime(x))$ for all $z\in\mathbb{C}$ and all $x$
\item The non-zero eigenvalues of $f^\prime(x)$, counting multiplicities, come in positive-negative pairs
\item $\tr(f^\prime(x)^{2k+1})=0$ for all $x$ and $k = 0,1,2,\ldots$
\end{enumerate}
\begin{proof}
$(i)\iff(ii)$: Assuming $(i)$, for every $x$, $p(z) = \det(I+zf^\prime(x))-\det(I-zf^\prime(x))$ is a polynomial in $z$ that is zero for infinitely many values of $z = h/2 \in \R_+$. By the Fundamental Theorem of Algebra, $p(z) = 0$ for all $z \in \mathbb{C}$. The converse follows from setting $h = 2z \in \R_+\subset\mathbb{C}$.

$(ii)\implies(iii)$: By triangularisation we can see that for every $x$, the polynomial $q(z) = \det(I-zf^\prime(x))$ is equal to $(1-z\lambda_1)\cdots(1-z\lambda_r)$ where $\lambda_1,\ldots,\lambda_r$ are the non-zero eigenvalues of $f^\prime(x)$. If $(i)$ holds, then $q(z) = q(-z)$, and the roots $1/\lambda_i$ of $q$ come in positive-negative pairs. Hence the eigenvalues $\lambda_i$ do too.

$(iii)\implies(iv)$: For all $x$, $\tr(f^\prime(x)^{2k+1}) = \lambda_1^{2k+1}+\cdots+\lambda_r^{2k+1}$ where $\lambda_1,\ldots,\lambda_r$ are the non-zero eigenvalues of $f^\prime(x)$. Hence if the non-zero eigenvalues come in positive-negative pairs then $\tr(f^\prime(x)^{2k+1}) = 0$ for $k=0,1,2,\ldots$.

$(iv)\implies(ii)$: Newton\rq{}s identity gives
\begin{equation}
e_{2k+1}(\lambda_1,\ldots,\lambda_n) = \frac{1}{2k+1}\sum_{i=1}^{2k+1} (-1)^{i-1} e_{2k+1-i}(\lambda_1,\ldots,\lambda_n)\tr(f^\prime(x)^i),
\end{equation}
where $e_j(\lambda_1,\ldots,\lambda_n)$ is the elementary symmetric polynomial in $\lambda_1,\ldots,\lambda_n$ and, incidentally, the coefficient of $z^j$ in $q(z) = \det(I-zf^\prime(x))$. Since for any $k$ and $i$, either $2k+1-i$ is odd or $i$ is odd, we can use an induction argument to show that all the coefficients of $z^{2k+1}$ in $q(z)$ are zero. Hence $q(z)=q(-z)$.
\end{proof}
\end{theorem}

\begin{corollary}
All elements of $\mathcal{D}$ are divergence free. Restricted to 2 dimensional vector fields, $\mathcal{D}$, $\mathcal{H}$ and $\mathcal{S}$ are all equal to divergence free vector fields.
\end{corollary}

\begin{theorem}
The set $\mathcal{F}^{(\infty)}$ contains all
\begin{enumerate}[(i)]
\item Affine vector fields $f(x) = Lx + d$ such that $\det(I+\frac{h}{2}L)=\det(I-\frac{h}{2}L)$ for all $h>0$
\item Vector fields such that $f^\prime(x)=JS(x)$ where $J$ is skew-symmetric and $S(x)$ is symmetric
\end{enumerate}
\begin{proof}
(i) Let $L$ satisfy the determinant condition \eqref{det}. By the Jordan normal form, and the fact that the eigenvalues must come in positive-negative pairs by Theorem \ref{dettheorem}, we can find an invertible matrix $P$ such that
\begin{equation*}
PLP^{-1} = \diag(\lambda_1I + N_1, -\lambda_1I + N_{-1}, \lambda_2I + N_2, -\lambda_{2}I + N_{-2},\ldots, \lambda_rI + N_r, -\lambda_r I + N_{-r}, N_0),
\end{equation*}
where the $N_k$ are matrices that are zero everywhere except for possible 1's on the first subdiagonal ${(N_k)}_{i+1,i}$. Hence $f$ is a tower of linear foliations of affine functions with Jacobian matrices either $N_0$ or $\diag(\lambda I + N_1, -\lambda I + N_{-1})$. If $f(x) = N_0 x + d$ then this is clearly a tower of foliations of zero systems i.e. $u(x) = 0$, $v(x,y) = x$. Now consider the case $f(x) = \diag(\lambda I + N_1, -\lambda I + N_{-1})x + d$. There is a simple permutation of variables so that the Jacobian matrix becomes
\begin{equation*}
\begin{pmatrix}
\lambda & 0 & 0 & 0 & 0 & 0 & 0 \\
0 & -\lambda & 0 & 0 & 0 & 0 & 0 \\
\star & 0 & \lambda & 0 & 0 & 0 & 0 \\
0 & \star & 0 & -\lambda & 0 & 0 & 0 \\
\vdots & \vdots & \ddots & \ddots &\ddots &\ddots & \vdots\\
0 & 0 & \cdots & \star & 0 & \lambda & 0 \\
0 & 0 & \cdots & 0 & \star & 0 & -\lambda
\end{pmatrix},
\end{equation*}
where the $\star$'s are possible 1's (0 otherwise). Hence $f$ is a tower of linear foliations of harmonic oscillators, $u(x_1,x_2) = (\lambda x_1, -\lambda x_2)$, $v(x_1,x_2,y_1,y_2) = (\star x_1 + \lambda y_1, \star x_2 -\lambda y_2)$.

(ii) By a linear orthogonal change of variables, we can assume $J = \diag(0,K^{-1})$, where $K$ is skew-symmetric. In this case there is symmetric $T(x)$ and $V(x)$ such that
\begin{equation}
f^\prime(x) = \begin{pmatrix} 0 & 0 \\ 0 & K^{-1} \end{pmatrix}\begin{pmatrix} T(x) & U(x) \\ U(x)^\top & V(x) \end{pmatrix} = \begin{pmatrix} 0 & 0 \\ K^{-1}U(x)^\top & K^{-1}V(x) \end{pmatrix}.
\end{equation}
This shows that $f$ possesses a linear foliation with a zero system $u\in \mathcal{H}$ and a system $v$ with $\partial_y v(x,y) = K^{-1}V(x)$ so that $y\mapsto v(x,y) \in \mathcal{H}$ with the same $P = K$ for all $x, y$.
\end{proof}
\end{theorem}

\begin{theorem}\label{foliationH}
Consider an $s$-stage SRK method that is volume preserving for the vector field $u : \R^m \to \R^m$, and let $v : \R^{m+n}\to\R^{m+n}$ be such that there exists an invertible matrix $P$ such that for all $x,y$,
\[
	P \partial_y v(x,y) P^{-1}=-\partial_y v(x,y)^\top.
	\]
Then the SRK method is volume preserving for the vector field
\begin{equation}\label{foliation_H}
	f(x,y) = (u(x), v(x,y))^\top.
\end{equation}
\begin{proof}
Let $A\in \R^{s\times s}$ and take $P$ from the assumption.
By Lemma \ref{SSRK} and Remark \ref{SRK}, a SRK method is volume preserving if 
\begin{equation}\tag{\ref{sstagecondition}}
\det(I_s\otimes I - h(A\otimes I)F) = \det(I_s\otimes I + h(A^\top \otimes I)F),
\end{equation}
where $F = \diag(f'(k_1),\ldots,f'(k_s))$.
For \eqref{foliation_H}, the Jacobian matrix becomes
\[
f'(x,y)=\begin{pmatrix} u'(x) & 0\\
\star & \partial_y v(x,y)\\
\end{pmatrix}
\]
and using a similarity transformation, we can bring 
$\det(I\otimes I - h(A\otimes I)F)$
to the form
\begin{equation}\label{detblock}
\det\begin{pmatrix}
I-a_{11}u_1' & \cdots & -a_{1s}u_s'& 0 &\cdots &0\\
\vdots & \ddots & \vdots & \vdots & \ddots & \vdots\\
-a_{s1}u_1' & \cdots & I-a_{ss}u_s' & 0 & \cdots & 0\\
\star & \ldots & \star & I-a_{11}v_1' & \cdots & a_{1s}v_s'\\
\vdots & \ddots & \vdots & \vdots & \ddots & \vdots\\
\star & \ldots & \star & I-a_{s1} v_1' & \cdots & a_{ss} v_s'\\
\end{pmatrix},
\end{equation}
where $u_i', v_i'$ are shorthand for $\partial_x u(k_i)$ and $\partial_y v(k_i)$, respectively.
Thus, the condition \eqref{sstagecondition} factorises to
\[
\det(I_s\otimes I - h(A\otimes I)U)\det(I_s\otimes I - h(A\otimes I)V)  = \det(I_s\otimes I + h(A^\top \otimes I)U)\det(I_s\otimes I + h(A^\top \otimes I)V),
\]
with
$U = \diag(u_1',\ldots,u'_s)$, $V = \diag(v_1',\ldots,v'_s)$. We compute
\begin{align*}
\det(I_s\otimes I - h(A\otimes I)V) &= \det(I_s \otimes I - h(I_s\otimes P)(A\otimes I)V(I_s\otimes P)^{-1}) \\
                                                  &= \det(I_s \otimes I - h(A\otimes I)(I_s\otimes P)V(I_s\otimes P^{-1})) \\
                                                  &= \det(I_s \otimes I + h(A \otimes I)V^\top) \\
                                                  &= \det(I_s \otimes I + h V(A^\top \otimes I)) \\
                                                  &= \det(I_s \otimes I + h (A^\top \otimes I)V).
\end{align*}
The last line comes from Sylvester\rq{}s determinant identity. 
The proof is completed noticing that $\det(I_s\otimes I - h(A\otimes I)U) = \det(I_s\otimes I + h(A^\top \otimes I)U)$ is satisfied by the assumption that the method is volume preserving for $u$.
\end{proof}
\end{theorem}

\begin{corollary}All SRK methods are volume preserving for vector fields in $\mathcal{F}^{(\infty)}$.
\begin{proof}
SRK methods are volume preserving for vector fields in $\mathcal{H}$ by Lemma \ref{Hlemma} and volume preservation for the recursive constructions of $\mathcal{F}^{(\infty)}$ is assured by Theorem~\ref{foliationH}.
\end{proof}
\end{corollary}

\begin{theorem}\label{foliationS}
Consider a SRK method with at most two stages (or a composition of such methods) that is volume preserving for the vector field $u : \R^m \to \R^m$, and let $v : \R^{m+n}\to\R^{m+n}$ be such that there exists an invertible matrix $P$ such that for all $x,y$,
\[
	P \partial_y v(x,y) P^{-1}=-\partial_y v(x,y).
	\]
Then the SRK method is volume preserving for the vector field
\begin{equation*}
	f(x,y) = (u(x), v(x,y))^\top.
\end{equation*}
\begin{proof}
Let $A\in\R^{2\times 2}$ and take $P$ from assumption. As in Theorem~\ref{foliationH}, the Jacobian matrix is block triangular 
\[
f'(x,y)=\begin{pmatrix} u'(x) & 0\\
\partial_x v(x,y) & \partial_y v(x,y)\\
\end{pmatrix}.
\]
For 2-stage methods, the condition for volume preservation from equation \eqref{2stagecond} is
\begin{multline*}
\det(I - ha_{11}f^\prime(k_1) - ha_{22}f^\prime(k_2) + h^2\det(A)f^\prime(k_1)f^\prime(k_2)) \\
= \det(I + ha_{11}f^\prime(k_1) + ha_{22}f^\prime(k_2) + h^2 \det(A) f^\prime(k_1) f^\prime(k_2)).
\end{multline*}
Now, because of the block-triangular structure of $f'(k_i)$ and
\[
f'(k_1)f'(k_2) = \begin{pmatrix} u^\prime(k_1)u^\prime(k_2) & 0 \\ \star & \partial_y v(k_1)\partial_y v(k_2) \\\end{pmatrix}, \quad f'(k_1)+f'(k_2) = \begin{pmatrix} u^\prime(k_1)+u^\prime(k_2) & 0 \\ \star & \partial_y v(k_1)+\partial_y v(k_2) \\\end{pmatrix},
\]
where we have used the convention that $u(k_i)$ has used the $x$ component of $k_i$. The condition \eqref{2stagecond} then factorises into
\begin{multline*}
\det(I - h(a_{11}f^\prime(k_1) +a_{22}f^\prime(k_2)) + h^2\det(A)f^\prime(k_1)f^\prime(k_2)) \\ =
\det(I - h(a_{11}u^\prime(k_1) +a_{22}u^\prime(k_2)) + h^2\det(A)u^\prime(k_1)u^\prime(k_2))\\ \cdot
\det(I - h(a_{11}\partial_y v(k_1) +a_{22}\partial_y v(k_2)) + h^2\det(A)\partial_y v(k_1)\partial_y v(k_2)).
\end{multline*}
A similarity transformation with $P$ leads to
\begin{multline}\label{2stagecondforv}
\det(I - h(a_{11}\partial_y v(k_1) +a_{22}\partial_y v(k_2)) + h^2\det(A)\partial_y v(k_1)\partial_y v(k_2)) \\
\begin{aligned}
&=\det(I - h(a_{11}P\partial_y v(k_1)P^{-1} + a_{22}P\partial_y v(k_2)P^{-1}) + h^2\det(A)P\partial_yv(k_1)P^{-1}P\partial_yv(k_2)P^{-1}) \\
&= \det(I + h(a_{11}\partial_y v(k_1) + a_{22}\partial_y v(k_2)) + h^2 \det(A) \partial_y v(k_1) \partial_y v(k_2)).
\end{aligned}
\end{multline}
Condition \eqref{2stagecond} for $f$ is now satisfied by considering \eqref{2stagecond} for the vector field $u$ (which holds because we assume the SRK method is volume preserving) and \eqref{2stagecondforv}.
This proves the result for 2-stage SRK methods. To complete the proof, note that a 1-stage SRK method is equivalent to a 2-stage SRK method with two equal stages, and compositions of volume preserving maps are also volume preserving.
\end{proof}
\end{theorem}

\begin{corollary}\label{cor31}All SRK methods with at most two stages (and compositions thereof) are volume preserving for vector fields in $\mathcal{F}^{(2)}$.
\begin{proof}
All SRK methods with at most two stages (and compositions thereof) are volume preserving for vector fields in $\mathcal{H}$ by Lemma \ref{Hlemma} and $\mathcal{S}$ by Lemma \ref{Slemma}. Volume preservation for the recursive construction of $\mathcal{F}^{(2)}$ is assured by Theorems~\ref{foliationH} and ~\ref{foliationS}.
\end{proof}
\end{corollary}

We already saw in the introduction that the implicit midpoint rule (which is the only 1-stage SRK method) is volume preserving for all $f \in \mathcal{D}$, and that \emph{all} such vector fields must lie in $\mathcal{D}$. However, does the set $\mathcal{F}^{(2)}$ contain \emph{all} vector fields such that all 2-stage SRK methods are volume preserving? And does the set $\mathcal{F}^{(\infty)}$ contain \emph{all} vector fields such that all SRK methods are volume preserving? We do not yet know the answers to these questions, but the following counterexamples are relevant.

The first counterexample shows that Corollary~\ref{cor31} is not true for three-stage methods. In the second example, we show that $\mathcal{D}\setminus \mathcal{F}^{(2)}\neq\emptyset$ and that only the midpoint rule can be volume preserving for all methods in $\mathcal{D}$. This counterexample is of the lowest possible dimension (3) but one might argue that volume preservation is hindered in this example because the vector field is not completely smooth at $x=0$.
The third example clarifies the matter: we give a way to construct a class of (smooth) vector fields in $\mathcal{D}$ for which two-stage methods cannot be expected to preserve volume.
\begin{example}
Hairer, Lubich and Wanner \cite[VI.9]{hairer2006geometric} used the vector field
\begin{equation*}
\dot{x} = \sin z, \quad \dot{y} = \cos z, \quad \dot{z} = \sin y + \cos x,
\end{equation*}
to show that the three-stage Gauss-Legendre method is not volume preserving, despite the vector field lying in $\mathcal{S}$. What could be interesting is to find some class of functions $\mathcal{F}^{(3)}$ such that all three-stage SRK methods are volume preserving, but not all four-stage SRK methods.
\end{example}

\begin{example}
Consider the continuously differentiable vector field
\begin{equation*}
f(x,y,z) = \left\{ \begin{array}{cc} (\frac13 x^3 - c, -x^2 y, 0 ) & \text{ if $x \geq 0$}\\ (\frac13 x^3 - c, 0, -x^2z) & \text{ if $x < 0$} \end{array} \right.,
\end{equation*}
\begin{equation*}
f'(x,y,z) = \begin{pmatrix} x^2 & 0 & 0 \\ -2xy & -x^2 & 0 \\ 0 & 0 & 0 \end{pmatrix} \text{ if $x\geq 0$}, \begin{pmatrix} x^2 & 0 & 0 \\ 0 & 0 & 0 \\ -2xz & 0 & -x^2 \end{pmatrix} \text{ if $x< 0$}.
\end{equation*}
Then $f\in\mathcal{D}$, but not all 2-stage SRK methods are volume preserving. The principle here is that if $k_1$ and $k_2$ have $x$-components with different signs, then $f^\prime(k_1)$ and $f^\prime(k_2)$ will violate the condition in \eqref{2stagecond}. 
For instance, the two-stage Gauss-Legendre method with initial value $(1/2,0,0)$, drift $c=1$ and step size $h=1/2$ has stage values with different signs in the $x$-coordinate and hence, does not preserve volume.
\end{example}
\begin{example}
The following example illustrates that SRK methods cannot preserve simple systems in $\mathcal{D}$ that do not belong to the classes $\mathcal{F}^{(\infty)}$ or $\mathcal{F}^{(2)}$. Let $g\in\mathcal{D}$ and let $A(x)$ be skew-symmetric (and invertible) and $S(y)$ be symmetric matrices. Then, any vector field with Jacobian matrix 
\[
	f'(x,y)=\begin{pmatrix} g'(x) & 0 \\ \star & A(x)S(y)\end{pmatrix}
\]
will satisfy the determinant condition, however, the similarity transform $P$ to yield $P\partial_y f(x,y)P^{-1}=-\partial_y f(x,y)^\top$ is now $P=A(x)^{-1}$ and this dependence on $x$ hinders a crucial step in the above proof.
For volume preservation of SRK methods, it is thus essential to have a constant transform $P$ for all values of $x,y$ or at least in a region of interest for the numerical integration.
We give the following concrete example,
\[
A(x)=\begin{pmatrix}
0 & x_1 & x_1 \\
-x_1 & 0 & x_1x_2 \\
-x_1 & -x_1x_2 & 0\\
\end{pmatrix},\quad
S(y)=\begin{pmatrix}
y_1^2 & 0 & 0 \\
0& y_2^2 & 0\\
0 & 0 & y_3\\
\end{pmatrix},
\]
which is combined with the harmonic oscillator $g(x_1,x_2)=(x_2,-x_1)^\top$ and could originate from
$f(x,y)=\left(g(x), A(x) (\frac13 y_1^3,\frac13 y_2^3, \frac12 y_1^2)^\top\right)^\top$.
Integrating with step size $h=1/2$ from $(x_0,y_0)=(1,1/2,1/3,1/4,1/5)$ leads to a change of volume for the two-stage Gauss-Legendre method. The implicit midpoint rule preserves volume as expected.
\end{example}
\section{Measure-preservation of Runge-Kutta methods}
\label{sec.measure}
In the introduction, we have pointed out that the trapezoidal method is not necessarily volume preserving but instead preserves the measure $\det\left(I-\tfrac{h}{2}f^\prime(x)\right)dx$ for quadratic Hamiltonian vector fields \cite{celledoni2013geometric}.
Recall that a map $\phi$ preserves a measure $\mu(x)dx$ if
\[
	\det(\phi'(x)) = \frac{\mu(x)}{\mu(\phi(x))}
\]
This result is generalised in the following lemma.
\begin{lemma}
The trapezoidal rule \eqref{trapezoidal} preserves the measure $\mu(x)dx$ with 
\begin{equation*}
\mu(x) = \det\left(I\pm\tfrac{h}{2}f^\prime(x)\right)
\end{equation*}
for vector fields $f$ that satisfy the determinant condition \eqref{det}.
\end{lemma}
\begin{proof}
We compute the Jacobian matrix $\phi'_h$ of the trapezoidal rule,
\[
\phi_h^\prime(x) = I + \frac{h}{2}f^\prime(x) + \frac{h}{2}f^\prime(\phi_h(x))\phi_h^\prime(x),
\]
and see that
\[
\det(\phi_h^\prime(x)) = \frac{\det(I+\frac{h}{2}f^\prime(x))}{\det(I-\frac{h}{2}f^\prime(\phi_h(x)))} = \frac{\mu(x)}{\mu(\phi_h(x))}.
\]
\end{proof}
This means that volume is conserved to order $\mathcal{O}(h^2)$ globally (by Theorem \ref{dettheorem}), but more importantly that the dynamics of the numerical solution lie in the class of measure preserving systems, giving a qualitative advantage over methods lacking this property \cite{mclachlan2001kinds}. The trapezoidal method is conjugate to the implicit midpoint rule \cite[VI.8]{hairer2006geometric}, which goes some way towards explaining this behaviour. However, the next method we consider has similar measure preserving properties, but doesn\rq{}t appear likewise to be \lq\lq{}conjugate to volume preserving\rq{}\rq{}.

There has been recent interest in the properties of the Kahan method \cite{celledoni2014integrability,celledoni2013geometric}. For a quadratic vector field $f(x) = Q(x) + L(x) + d$ where $Q$ is quadratically homogeneous, $L$ is linear and $d$ is constant, the symmetric bilinear form $q(x,y)$ is formed by polarisation,
\begin{equation}
q(x,y) = \frac12 \Bigl(Q(x+y) - Q(x) - Q(y)\Bigr),
\end{equation}
and Kahan\rq{}s unconventional numerical method is then given by
\begin{equation}\label{eqkahan}
\frac{\phi_h(x)-x}{h} = q(x,\phi_h(x)) + \frac12L\,\Bigl(x+\phi_h(x)\Bigr) + d.
\end{equation}
In \cite{celledoni2013geometric}, it was shown that Kahan\rq{}s method is equivalent to a three-stage Runge-Kutta method restricted to quadratic vector fields. We give the following generalisation.
\begin{lemma}\label{lemmaKahanRK}
Restricted to quadratic vector fields, Kahan\rq{}s method is equivalent to the $s$-stage Runge-Kutta method
\begin{equation}\label{kahanRK}
\phi_h(x) = x + h\sum_{i=1}^s b_i f(x + c_i(\phi_h(x)-x)),
\end{equation}
for any $b$ and $c$ satisfying $\sum_{i=1}^s b_i = 1$, $\sum_{i=1}^s b_i c_i = \frac12$, $\sum_{i=1}^s b_ic_i^2 =0$. This implies that the Butcher tableau satisfies $A = cb^\top$, but the converse is not true.
\end{lemma}
\begin{proof}
Let $x'=\phi_h(x)$ and write the vector field as $f(x) = q(x,x) + Lx + d$ with the symmetric bilinear form $q$, then, expanding out and setting equal to Kahan's method \eqref{eqkahan}
\begin{align*}
	 \frac{x'-x}{h} &= \sum_{i=1} b_i q\Bigl(x + c_i(x'-x),\ x + c_i(x'-x)\Bigr) + L (x + c_i(x'-x)) + d\\
	%
	&=q(x',x) + \tfrac12L(x+x')+d
	\end{align*} 
yields the above conditions.
\end{proof}
In \cite[Prop. 5]{celledoni2013geometric}, it was shown that for quadratic Hamiltonian vector fields, Kahan\rq{}s method preserves the measure with density $
\mu(x) = \det(I-\frac{h}{2}f^\prime(x))^{-1}.$
The proof is easily extended to all quadratic vector fields satisfying  the determinant condition \eqref{det}.
\begin{lemma}
Kahan's method preserves the measure $\mu(x)dx$ with 
\begin{equation}\label{kahanmeasure}
\mu(x) = \det\left(I\pm\tfrac{h}{2}f^\prime(x)\right)^{-1}
\end{equation}
for \emph{quadratic} vector fields $f$ that satisfy the determinant condition \eqref{det}.
\end{lemma}
\begin{proof}
We compute the Jacobian matrix $\phi'_h$ of Kahan's method in the form \eqref{kahan},
\[
\phi_h^\prime(x) = \frac{I - \frac{h}{2}f^\prime(x) + \frac{h}{2}f^\prime\left(\frac{x+\phi_h(x)}{2}\right)}{I + \frac{h}{2}f^\prime(\phi_h(x)) - \frac{h}{2}f^\prime\left(\frac{x+\phi_h(x)}{2}\right)}.
\]
Since $f$ is quadratic, $f'$ is affine and thus
\[
\det(\phi_h^\prime(x)) = \frac{\det(I+\frac{h}{2}f^\prime(\phi_h(x)))}{\det(I-\frac{h}{2}f^\prime(x))} = \frac{\mu(x)}{\mu(\phi_h(x))}.
\]
\end{proof}

 Due to the similarity of this measure to that preserved by the trapezoidal method, one might at first glance suggest that the Kahan method is conjugate to some volume preserving method too, but this does not appear to be the case. At least, Kahan\rq{}s method is not conjugate by B-series to any symplectic method \cite{celledoni2013geometric}. It may be interesting to investigate how these measure preserving properties of the trapezoidal rule and Kahan\rq{}s method can be generalised.

From Lemmata \ref{linearchangeofvariables}, \ref{foliationlemma} and \ref{foliationsum} on linear foliations follow similar measure preservation properties generalising the volume preservation properties discussed in the previous section.
\begin{theorem}
Suppose that a given Runge-Kutta method preserves the measure $\mu$ on $\R^n$ when solving the ODE $\dot{x} = f(x)$. Then when solving the ODE $\dot{x} = \tilde{f}(x)$, where $\tilde{f}(x) = Pf(P^{-1}x)$ for some invertible matrix $P$, the method preserves the measure with density $\tilde\mu(x)=\mu(P^{-1}x)$.
\begin{proof}
By assumption, $\det(\phi_h^\prime(y))\mu(\phi_h(y)) = \mu(y)$ for all $y\in\R^n$. Using the notation and results of Lemma \ref{linearchangeofvariables}, $\det(\tilde{\phi}_h^\prime(x))\tilde{\mu}(\tilde\phi_h(x)) = \det(\phi_h^\prime(P^{-1}x))\mu(\phi_h(P^{-1}x)) = \mu(P^{-1}x) = \tilde\mu(x)$.
\end{proof}
\end{theorem}
\begin{theorem}
Suppose that a given 1-stage Runge-Kutta method preserves the measure $\rho dx$ on $\R^m$ when solving the ODE $\dot{x}=u(x)$, and it preserves the measure $\nu(y) dy$ on $\R^n$ when solving the ODE $\dot{y} = v(x,y)$ for all $x \in \R^m$. 
Then when solving the ODE $(\dot{x},\dot{y}) = (u(x),v(x,y))$, the method preserves the product measure $\mu(x,y) dxdy = \rho(x)\nu(y)dx dy$ on $\R^{n+m}$. 
\begin{proof}
By Lemma \ref{foliationlemma}, {$\phi_h(x,y) = (\psi_h(x),\chi_h(k_1(x),y))^\top$}, where $k_1(x)$ is the internal stage of the 1-stage method. Hence by the definition of $\mu$,
\begin{equation*}
\mu(\phi_h(x,y)) = \rho(\psi_h(x))\nu(\chi_h(k_1(x),y)).
\end{equation*}
By assumption, we have for all $x$ and $y$,
\begin{equation}
\det(\psi_h^\prime(x))\rho(\psi_h(x)) = \rho(x), \quad \det(\partial_y \chi_h(x,y))\nu(\chi_h(x,y)) = \nu(y).
\end{equation}
Finally, Lemma \ref{foliationlemma} gives us $\det(\phi_h'(k_1(x),y)) = \det(\psi_h'(x))\det(\partial_y\chi_h(x,y))$. Combining all of these results,
\begin{align*}
\det(\phi_h^\prime(x,y))\mu(\phi_h(x,y)) &= \det(\psi_h^\prime(x))\rho(\phi_h(x))\det(\partial_y\chi_h(k_1(x),y))\nu(\chi_h(k_1(x),y)) \\
                                                                &= \rho(x)\nu(y)\\
                                                                &= \mu(x,y).
\end{align*}
Hence the measure with density $\mu$ on $\R^{n+m}$ is conserved.
\end{proof}
\end{theorem}
From the results of Lemma~\ref{foliationsum} and using its notation, we deduce that a generalization for measure preserving RK methods with more stages even for sums $f(x,y)=(u(x),w(x)+v(y))^\top$ is not trivial since then,
\[\det(\phi_h'(x,y))=\det(\psi_h'(x))\det(\partial_y\chi_h(d(x),y+he(x))),\]
and a product measure $\mu(x,y)dxdy=\rho(x)\nu(y)dxdy$ transforms according to
\begin{align*}
\det(\phi_h'(x,y))\mu(\phi_h(x,y))&=\det(\phi_h'(x,y))\rho(\psi_h(x))\  \nu(\chi_h(d(x),y+he(x)) + h c(x))\\
&= \det(\psi_h'(x))\rho(\psi_h(x))\cdot\det\Bigl(\partial_y\chi_h(d(x),y+he(x)) \Bigr)\ \nu\Bigl(\chi_h(d(x),y+he(x)) + h c(x)\Bigr).
\end{align*}
Assume that $\psi_h,\chi_h$ preserve the measures with densities $\rho(x)$ and $\nu(y)$, respectively, then, 
if $c_h=e_h=0$, the product measure is preserved.
This additional condition holds, e.g., for the trapezoidal rule for which we get that $d_h(x)=(w(k_1)+w(k_2))/2$. Further methods satisfying $c_h=e_h=0$ can be constructed easily\footnote{Let, e.g., $a_{1j}=0$ and $a_{ij}=b_j$ for some $i$ and all $j$.} but they might preserve measures for trivial vector fields only.
Kahan's method derived from Lemma~\ref{lemmaKahanRK} does not simplify in this way, however, we can give the following result:
\begin{theorem}\label{kahanfoliation}
Generalized Kahan's methods from Lemma~\ref{lemmaKahanRK} preserve the measure $\mu(x,y)dxdy$ with $\mu(x,y)=\det(I+\frac{h}{2}f'(x,y))^{-1}$ for linearly foliate vector fields of the form $f(x,y)=(u(x),v(y)+w(x))^\top$ where $w$ is arbitrary, and $u,v\in\mathcal{D}$ are quadratic.
\begin{proof}
Let $z=(x,y)$, and write $\phi_h(z)=(\psi_h(x),\sigma_h(x,y))^\top$. We compute the Jacobian determinant of \eqref{kahanRK}
\begin{align*}\nonumber
\det(\phi_h'(z)) &=\frac{\det(I + h \sum_{i=1}^{N}b_i(1-c_i)f'(z+c_i(\phi_h(z)-z)))}{\det(I - h \sum_{i=1}^{N}b_ic_if'(z+c_i(\phi_h(z)-z))))}
\intertext{using that $f'$ is block-diagonal, we arrive at}\nonumber
&=\frac{\det(I + h \sum_{i=1}^{N}b_i(1-c_i)u'(x + c_i(\psi_h(x)-x)))}{\det(I - h \sum_{i=1}^{N}b_ic_iu'(x + c_i(\psi_h(x)-x)))}
\frac{\det(I + h \sum_{i=1}^{N}b_i(1-c_i)v'(y + c_i(\sigma_h(x,y)-y)))}{\det(I - h \sum_{i=1}^{N}b_ic_iv'(y + c_i(\sigma_h(x,y)-y)))}\\
\intertext{and since $u',v'$ are affine, we can simplify using the assumptions on the coefficients from Lemma~\ref{lemmaKahanRK} to}
\nonumber
&=\frac{\det(I + \frac{h}{2} u'(\psi_h(x)))}{\det(I - \frac{h}{2} u'(x)) }\frac{\det(I + \frac{h}{2} v'(\sigma_h(x,y)))}{\det(I - \frac{h}{2} v'(y))}
=
\frac{\det(I + \frac{h}{2} f'(\phi_h(x,y)))}{\det(I - \frac{h}{2} f'(x,y))}
= \frac{\mu(z)}{\mu(\phi_h(z))}.
\end{align*}
\end{proof}
\end{theorem}

\begin{remark}
The theorem is not true for more general foliate vector fields within the class $\mathcal{F}^{(\infty)}$, e.g., $f( x,y)=(u(x),J^{-1}\nabla_y H(x,y))^\top$ where $u(x)$ is a simple harmonic oscillator and with the Hamiltonian $H(x,y)=(p_xq_x) p_y q_y$ using the usual notation for the momentum and position coordinates $x=(q_x,p_x)$, $y=(q_y,p_y)$. Note that the Hamiltonian is still quadratic in $y$!
\end{remark}

\section*{Acknowledgements}
The authors would like to thank Robert McLachlan for stimulating discussions and suggestions. This research was supported by a Marie Curie International Research Staff Exchange Scheme Fellowship within the 7th European Community Framework Programme; by the Australian Research Council; and by the UK EPSRC grant EP/H023348/1 for the Cambridge Centre for Analysis.

\section*{References}
\bibliographystyle{model1-num-names}
\bibliography{VPRKbib}

\end{document}